\documentclass[11pt,twoside]{amsart}
\usepackage{latexsym}
\usepackage{graphics}
\usepackage[pdftex]{graphicx}
\usepackage{multirow}
\usepackage{caption}
\usepackage{subcaption}
\usepackage{amsmath, amsthm, amscd, amsfonts, amssymb, float, graphicx, color, xcolor, soul}
\usepackage{tikz}
\input xy
\xyoption{all}
\newtheorem{lemma}{Lemma}
\newtheorem{theorem}[lemma]{Theorem}
\newtheorem{proposition}[lemma]{Proposition}

\newtheorem*{theorem*}{Theorem}
\newtheorem*{conj*}{Conjecture}

\def\deg{{\rm deg}}

\def\Aut{{\rm Aut}}
\def\Out{{\rm Out}}
\def\Inn{{\rm Inn}}
\def\Sol{{\rm Sol}}
\def\PSL{{\rm PSL}}
\def\PGL{{\rm PGL}}
\def\U{{\rm U}}

\def\Cay{{\rm Cay}}
\def\soc{{\rm soc}}
\def\cos{{\rm Cos}}

\def\BiCay{{\rm BiCay}}

\def\Core{{\rm Core}}

\def\pn{\par\noindent}
\begin{document}
\openup 1.8\jot
\title{Tetravalent vertex-transitive graphs of order $6p$}
\author{Majid Arezoomand, Mohsen Ghasemi* and Mohammad A. Iranmanesh}

\thanks{{\scriptsize
\hskip -0.4 true cm MSC(2010): Primary: 20B20, Secondary: 05C25.
\newline Keywords: Tetravalent graph, vertex transitive graph, non-Cayley graph.\\
$*$Corresponding author}}

\begin{abstract}
A graph is vertex-transitive if its automorphism group acts
transitively on vertices of the graph. A vertex-transitive graph is
a Cayley graph if its automorphism group contains a subgroup acting
regularly on its vertices. In this paper, the tetravalent
vertex-transitive non-Cayley graphs of order $6p$ are classified
for each prime $p$.
\end{abstract}

\maketitle
\section{Introduction}
In this paper we consider undirected finite connected graphs without
loops or multiple edges. For a graph $X$ we use $V(X)$, $E(X)$,
$A(X)$ and $\Aut(X)$ to denote its vertex set, edge set, arc set and
its full automorphism group, respectively. For $u, v\in V(X)$, $u
\thicksim v$ represents that $u$ is adjacent to $v$, and is denoted by
$\{u, v\}$ the edge incident to $u$ and $v$ in $X$, and $N_X(u)$ is
the neighborhood of $u$ in $X$, that is, the set of vertices
adjacent to $u$ in $X$. A graph $X$ is said to be {\em
$G$-vertex-transitive}, {\em $G$-edge-transitive} and {\em $G$-arc-transitive}
(or {\em $G$-symmetric}) if $G\leq\Aut(X)$ acts transitively on  $V(X)$,
$E(X)$ and $A(X)$, respectively. In the special case, if $G=\Aut(X)$ then $X$ is said to be {\em
vertex-transitive}, {\em edge-transitive} and {\em arc-transitive}
(or {\em symmetric}). An {\em s-arc} in a graph $X$ is
an ordered $(s+1)$-tuple $(v_0,v_1,\cdots ,v_s)$ of vertices of $X$
such that $v_{i-1}$ is adjacent to $v_i$ for $1\leq i\leq s$, and
$v_{i-1}\neq v_{i+1}$ for $1\leq i\leq s$; in other words, a
directed walk of length $s$ which never includes a backtracking. A
graph $X$ is said to be {\em $s$-arc-transitive} if $\Aut(X)$ is
transitive on the set of $s$-arcs in $X$. A subgroup of the
automorphism group of a graph $X$ is said to be {\em $s$-regular} if
it acts regularly on the set of $s$-arcs of $X$. Recall that a permutation group $G$ acting on
a set $\Omega$ is called semiregular if the stabilizer of $\alpha\in G$, $G_{\alpha}=1$ for all
$\alpha\in G$ and is called regular if it is semiregular and transitive.

Let $G$ be a finite group and $S$ be a subset of $G$ such that $1
\notin S$ and $S=S^{-1}$ where $S^{-1}=\{s^{-1}\mid s \in S\}$. The {\em Cayley
graph} $X=\Cay(G,S)$ on $G$ with respect to $S$ is defined as the
graph with vertex set $V(X)=G$ and edge set $E(X)=\{\{g, sg\}
\mid g \in G, s\in S\}$. The automorphism group $\Aut(X)$ of $X$
contains the right regular representation $R(G)$ of $G$, the
acting group of $G$ by right multiplication, as a subgroup, and
$R(G)$ is regular on $V(X)$. A Cayley graph, $\Cay (G,S)$ is said to be {\em normal} if $R(G)$ is normal in
$\Aut(\Cay (G,S))$. By \cite[Proposition 3.1]{Xu1}, $N_{\Aut(\Cay (G,S))}(R(G))=R(G)\Aut (G,S)$,
where $\Aut (G,S)=\{\alpha\in\Aut(G)\mid S^\alpha=S\}$. Note that $\Aut(X)$ is normal if and only
if $\Aut (G,S)=\Aut(X)_v$ for some $v\in V(X)$.

It is well known that a vertex-transitive graph is a Cayley graph if
and only if its automorphism group contains a subgroup acting
regularly on its vertex set (see, for example,~\cite[Lemma 4]{S}).
There are vertex-transitive graphs which are not Cayley graphs and
the smallest one is the well-known Petersen graph. Such a graph will
be called a {\em vertex-transitive non-Cayley graph}, or a
{$\mathcal{VNC}$-graph} for short.

Many publications have been put into service of investigation the $\mathcal{VNC}$-graphs from different
perspectives. For example, in~\cite{M.1}, Maru\v si\v c asked for a
determination of the set $\mathcal{NC}$ of non-Cayley numbers, that is,
those numbers $n$ for which there exists a $\mathcal{VNC}$-graph of order $n$, and
to settle this question, a lot of $\mathcal{VNC}$-graphs were constructed in
~\cite{HIP,IP,L,M.2,M.3,M.4,M.5,MC,MC1,Mill,Seress}.
In~\cite{Feng}, Feng considered the question to determine the
smallest valency for $\mathcal{VNC}$-graphs of a given order and it was solved
for the graphs of odd prime power order. From~\cite{CO, Ma} all
cubic or tetravalent $\mathcal{VNC}$-graphs of order $2p$ are known for each
prime $p$. In~\cite{Z1,Z2,Z3}, all cubic $\mathcal{VNC}$-graphs of order
$2pq$ were classified for any primes $p$ and $q$. Also tetravalent
$\mathcal{VNC}$-graphs of order $2p^2$ and $4p$ were classified for any prime $p$
(see~\cite{CGhQ,Z}). We know that all components of vertex-transitive graphs are isomorphic.
So without loss of generality we may assume that $X$ is connected.

A graph $X$ is said to be a {\em bi-Cayley graph} over a group $G\leq \Aut(X)$ if it admits $G$ as a
semiregular automorphism group with two orbits of equal size. Note that every bi-Cayley
graph admits the following concrete realization. Let $R, L$ and $S$ be subsets of
a group $G$ such that $R=R^{-1}$, $L=L^{-1}$ and $R\cup L$ does not contain
the identity element of $G$. Let $g_0=(g, 0)$ and $g_1=(g, 1)$ where $g\in G$. Define the graph $\BiCay(G, R, L, S)$
to have vertex set the union of the {\em right part} $G_0=\{g_0 \mid g\in G\}$ and the {\em left part}
$G_1=\{g_1\mid g\in G\}$ and edge set the union of the {\em right edges} $\{\{h_0, g_0\}\mid gh^{-1}\in R\}$,
the {\em left edges} $\{\{h_1, g_1\} \mid gh^{-1} \in L\}$ and the {\em spokes}
$\{\{h_0, g_1\}\mid gh^{-1}\in S\}$. Now we introduce the coset graph constructed from a finite
group $G$ relative to a subgroup $H$ of $G$ and a union $D$ of some double cosets of $H$ in
$G$ such that $D^{-1}=D$. The {\em coset graph} $\cos(G, H, D)$ of $G$ with respect to
$H$ and $D$ is defined to have vertex set $[G:H]$, the set of right cosets of $H$ in $G$,
and edge set $\{\{Hg, Hdg\}\mid g\in G, d\in D\}$. The graph $\cos(G, H, D)$ has
valency $|D|/|H|$ and is connected if and only if $D$ generates the group $G$. The action
of $G$ on $V(\cos(G, H, D))$ by right multiplication induces a vertex-transitive
automorphism group, which is arc-transitive if and only if $D$ is a single double coset.
Clearly, $\cos(G, H, D)\cong\cos(G, H^{\alpha}, D^{\alpha})$ for every $\alpha\in\Aut(G)$.
Conversely, let $X$ be a graph and let $A$ be a vertex-transitive subgroup of $\Aut(X)$.
By \cite{S1}, the graph $X$ is isomorphic to a coset graph $\cos(A, H, D)$, where $H=A_u$
is the stabilizer of $u \in V(X)$ in $A$ and $D$ consists of all elements of $A$ which
map $u$ to one of its neighbors. It is easy to show that $\Core_A(H)=1$ and that $D$ is a union
of some double cosets of $H$ in $A$ satisfying $D=D^{-1}$. Recall that a graph $\Gamma=(V , E)$
is called an {\em $(m,n)$-metacirculant}, where $m,n$ are positive integers,
if $\Gamma$ is of order $|V|= mn$ and has two automorphisms $\rho, \sigma$ such that
\begin{itemize}
\item [$(a)$] $\langle\rho\rangle$ is semiregular and has $m$ orbits on $V$,
\item [$(b)$] $\sigma$ cyclically permutes the $m$ orbits of $\langle\rho\rangle$ and normalizes $\langle\rho\rangle$, and
\item [$(c)$] $\sigma^m$ fixes at least one vertex of $\Gamma$.
\end{itemize}

A graph $\widetilde{X}$ is called a {\em covering} of a graph $X$ with
projection $\wp:\widetilde{X}\rightarrow X$ if there is a surjection
$\wp:V{(\widetilde{X})}\rightarrow V(X)$ such that
$\wp|_{N_{\widetilde{X}}({\tilde{v}})}:{N_{\widetilde{X}}({\tilde{v}})}\rightarrow
{N_{X}(v)}$ is a bijection for any vertex $v\in V(X)$ and
$\tilde{v}\in \wp^{-1}(v)$. A covering $\widetilde{X}$ of $X$ with a
projection $\wp$ is said to be {\em regular} (or {\em $K$-covering}) if
there is a semiregular subgroup $K$ of the automorphism group
$\Aut(\widetilde{X})$ such that the graph $X$ is isomorphic to the
quotient graph $\widetilde{X}/K$, say by $\Im$, and the quotient map
$\widetilde{X}\rightarrow \widetilde{X}/K$ is the composition $\wp \Im$
of $\wp$ and $\Im$ (for the purpose of this paper, all functions are
composed from left to right). If $K$ is cyclic or elementary abelian
then $\widetilde{X}$ is called a {\em cyclic} or an {\em elementary
abelian covering} of $X$, and if $\widetilde{X}$ is connected $K$
becomes the covering transformation group. The {\em fibre} of an
edge or a vertex is its preimage under $\wp$. An automorphism of
$\widetilde{X}$ is said to be {\em fibre-preserving} if it maps a
fibre to a fibre, while every covering transformation maps a fibre
on to itself. All of fibre-preserving automorphisms form a group
called the {\em fibre-preserving group}.

Let $\widetilde{X}$ be a $K$-covering of $X$ with a projection $\wp$.
If $\alpha\in\Aut(X)$ and $\widetilde{\alpha}\in\Aut(\widetilde{X})$ satisfy $\widetilde{\alpha}\wp=\wp\alpha$, we call
$\widetilde{\alpha}$ a {\em lift} of $\alpha$, and $\alpha$ the {\em
projection} of $\widetilde{\alpha}$. Concepts such as a lift of a
subgroup of $\Aut(X)$ and the projection of a subgroup of
$\Aut(\widetilde{X})$ are self-explanatory. The lifts and the
projections of such subgroups are of course subgroups in
$\Aut(\widetilde{X})$ and $\Aut(X)$, respectively.

Let $m_1, m_2>1$ be two odd integers such that $(m_1,m_2)=1$. Let $1\leq t\leq m_2$, $(t,m_2)=1$ be such
that $t^2\equiv -1\pmod{m_2}$. Let $H=\langle r\rangle\times\langle s\rangle\cong\mathbb{Z}_{m_1}\times\mathbb{Z}_{m_2}=\mathbb{Z}_{m_1m_2}$.
Set $R=\{r,r^{-1}\}$, $L=\{r^t,r^{-t}\}$ and $S=\{1,s\}$ and $X_{m_1,m_2,t}=\BiCay(H,R,L,S)$. Then we have the following result.
\begin{proposition}\cite[Theorem 4.3]{QZ}\label{m1m2t}\label{p3} Let $X$ be a connected tetravalent vertex-transitive bi-Cayley graph over a cyclic
group of odd order $n$. Then $X$ is a $\mathcal{VNC}$-graph if and only if $X\cong X_{m_1,m_2,t}$ for some integers $m_1,m_2$ and $t$.
\end{proposition}
The following theorem is the main result of this paper.

\begin{theorem}\label{a3}
Let $X$ be a connected tetravalent vertex-transitive graph of order $6p$, where $p$
is a prime. Then $X$ is a $\mathcal{VNC}$-graph if and only if $X\cong X_{3,p,t}$, where $1\leq t\leq p-1$ and $t^2\equiv -1$ (mod $p$)
or $X$ is one of the nine specified graphs.
\end{theorem}

\section{Preliminaries}
In this section, we introduce some notation and definitions as well
as some preliminary results which will be used later in the paper.

For a regular graph $X$, use $d(X)$ to represent the valency of $X$,
and for any subset $B$ of $V(X)$, the subgraph of $X$ induced by $B$
will be denoted by $X[B]$. Let $X$ be a connected vertex-transitive
graph, and let $G\leq\Aut(X)$ be vertex-transitive on $X$.
For a $G$-invariant partition $\Omega$ of $V(X)$, the {\em quotient graph}
$X_\Omega$ is defined as the graph with vertex set $\Omega$ such that, for
any two vertices $B, C\in\Omega$, $B$ is adjacent to $C$ if and only if
there exist $u\in B$ and $v\in C$ which are adjacent in $X$. Let $N$
be a normal subgroup of $G$. Then the set $\Omega$ of orbits of $N$ in
$V(X)$ is a $G$-invariant partition of $V(X)$. In this case, the symbol $X_\Omega$ will be replaced by $X_N$.

For two groups $M$ and $N$, by $N\rtimes M$ we denote the semidirect product of $N$
by $M$. For a group $G$, the largest solvable normal subgroup of $G$ is called the {\em solvable radical}
of $G$ and is denoted by $\Sol(G)$. The outer automorphism group of the group $G$, is the quotient,
$\Aut(G)/ \Inn(G)$, where $\Aut(G)$ is the automorphism group of $G$ and $\Inn(G)$
is the subgroup consisting of inner automorphisms. The outer automorphism group is usually denoted by $\Out(G)$.
Also a lexicographic product of two graphs $X$ and $Y$ which is denoted by $X[Y]$
is defined as the graph with vertex set $V(X)\times V(Y)$ such that for any  two
vertices $u=(x_1, y_1)$ and $v=(x_2, y_2)$ in $X[Y]$, $u$ is adjacent to $v$ in $X[Y]$, whenever either
$\{x_1, x_2\} \in E(X)$ or $x_1=x_2$ and $\{y_1, y_2\} \in E(Y)$. Note that the lexicographic product of
two Cayley graphs is a Cayley graph. For group and
graph-theoretic terms not defined here we refer the reader to \cite{B} and \cite{W}, respectively.
Now we state the following well known result.
\begin{proposition}\label{p2.2}
Let $X$ be a $G$-vertex-transitive graph. Then $X$ is symmetric if and only if each vertex stabilizer
$G_v$ acts transitively on the set of vertices adjacent to $v$, where $v\in V(X)$.
\end{proposition}

A finite simple group $G$ is said to be a $K_3$-group if its order has
exactly three distinct prime divisors. By \cite[Pages 12-14 ]{G}, $G$ is isomorphic to one of the following
groups:
\[ A_5, A_6, \PSL_2(7), \PSL_2(8),
\PSL_2(17), \PSL_3(3), \U_3(3), \U_4(2).\tag{1}
\]
The {\em socle} of a group $G$ is the subgroup generated by the set of all minimal normal subgroups of
$G$, it is denoted by $\soc(G)$. Also a group $G$ is said to be {\em almost simple} if
$T \leq G \leq \Aut(T)$, where $T$ is a non-abelian simple group. It is well known
that $G$ is an almost simple group if and only if $\soc(G)=T$ for some non-abelian simple group $T$.

Now we prove the following lemma.
\begin{lemma}\label{a1}
Let $G$ be an almost simple group and $\soc(G)=A$, where $A$ is a non-abelian simple $K_3$-group. Then
\begin{itemize}
\item[$(i)$] If $A\cong A_5$ then $G\cong A_5$ or $S_5$.
\item[$(ii)$] If $A\cong A_6$ then $G\cong A_6$ or $S_6$ or $S_6\rtimes\mathbb{Z}_2$.
\item[$(iii)$] If $A\cong\PSL_2(p)$, where $p\in \{7, 17\}$, then $G\cong\PSL_2(p)$ or $\PGL_2(p)$.
\item[$(iv)$] If $A\cong\PSL_2(8)$  then $G\cong\PSL_2(8)$ or $G/\PSL_2(8)\cong\mathbb{Z}_3$.
\item[$(v)$] If $A\cong\PSL_3(3)$ then $G\cong\PSL_3(3)$ or $G/\PSL_3(3)\cong\mathbb{Z}_2$.
\item[$(vi)$] If $A\cong\U_3(3)$ then $G\cong\U_3(3)$ or $G/\U_3(3)\cong\mathbb{Z}_2$.
\item[$(vii)$] If $A\cong\U_4(2)$ then $G\cong\U_4(2)$ or $G/\U_4(2)\cong\mathbb{Z}_2$.
\end{itemize}
\end{lemma}

\begin{proof}
Since $G$ is an almost simple group and $\soc(G)=A$, it implies that $A\leq G\leq\Aut(A)$.
If $A\cong A_5$, then $A_5\leq G\leq \Aut(A_5)$. Now
since $|\Out(A_5)|=2$, it follows that $G \cong A_5$ or $S_5$ and $(i)$ holds. Also if
$A$ is isomorphic to one of $\PSL_2(7)$, $\PSL_2(17)$, $\PSL_3(3)$, $\U_3(3)$
or $\U_4(2)$ then $|\Out(A)|=2$ and the assertions in $(iii),(v),(vi)$ and
$(vii)$ hold. If $A\cong A_6$ then $|\Out(A_6)|=4$ and
$\Aut(A_6)\cong S_6\rtimes \mathbb{Z}_2$. Thus $G\cong A_6$ or $S_6$ or
$S_6\rtimes \mathbb{Z}_2$ and $(ii)$ holds. Finally if $A\cong\PSL_2(8)$
then $|\Out(\PSL_2(8))|=3$. Hence $G\cong\PSL_2(8)$ or $G/ \PSL_2(8)\cong\mathbb{Z}_3$ and $(iv)$ holds.
\end{proof}

\section{Main Results}
In this section we classify all connected tetravalent $\mathcal{VNC}$-graphs of order $6p$ where $p$ is a prime.
To do this we prove the following propositions.
\begin{proposition}\label{p2}
Let $X$ be a graph,  $A=\Aut(X)$ and $N$ be the  normal subgroup of $A$. Also let $\Omega$ be the
set of orbits of $N$ on $V(X)$ and $X_N$ be a Cayley graph on subgroup of $A/K$, where $K$ is the
kernel of action $N$ on $\Omega$. Then $X$ is a Cayley graph on some subgroup of $A$.
\end{proposition}
\begin{proof}
Suppose that $X_N=\Cay(H, S)$, where $H\leq A/K$. Thus we may suppose that $H=T/K$, where $T \leq A$.
Let $x, y$ be two arbitrary elements of $V(X)$. Set $X=x^N$ and $Y=y^N$. If $X=Y$ then there is $\alpha\in N$
so that $x^{\alpha}=y$. Thus we may suppose that $X\neq Y$. By our assumption there is an element
$K\beta\in T/K$ such that $(x^N)^{K\beta}=(y^N)$ and so $(x^N)^{\beta}=(y^N)$. Now $x^{\beta}=y^n$ for
some $n\in N$. Now $x^{({\beta}n^{-1})}=y$, where ${\beta}n^{-1} \in T$. Thus $T$ acts transitively
on $V(X)$. Suppose that $x^t=x$, where $t\in T$. Now $(x^N)^{Kt}=(x^N)$, a contradiction. Thus $T$
acts regularly on $V(X)$ and $X$ is a Cayley graph on $T$.
\end{proof}

\begin{proposition}\label{p3}
Let $X$ be a connected tetravalent vertex-transitive graph of order $3p$, where $p$ is a prime.
Then $X$ is a $\mathcal{VNC}$-graph if and only if $X$ is isomorphic to $L(O_3)$, the line graph of the Petersen graph.
\end{proposition}
\begin{proof}
Suppose that $X$ is a connected tetravalent $\mathcal{VNC}$-graph of order $3p$. By \cite{MR}, all tetravalent
vertex-transitive graphs of order $6 , 9$ and $21$ are Cayley graphs. Also a tetravalent vertex-transitive graph
of order $15$ is not a Cayley graph and is isomorphic to $L(O_3)$.
Thus we may assume that $p\geq 11$. If $X$ is symmetric then by \cite[Page 215]{WX}, is isomorphic to
one of $L_3(2)^4_{21}$ or $G(3p, 2)$, which are Cayley graphs.
Thus in the following we may assume that $X$ is not symmetric. Let $A=\Aut(X)$ and $v\in V(X)$.
Since $A_v$ is a $\{2, 3\}$-group we have $|A|=2^s3^tp$, for some non-negative integers $s, t$. We claim that $A$ is not solvable.

Suppose by contrary that $A$ is solvable and let $N$ be a minimal normal subgroup of $A$. Then $N$ ia an
elementary abelian $r$-group, where $r\in \{2, 3, p\}$, $N=T^k$, and $T$ is a simple group.
Let $\Omega$ be the set of orbits of $N$ on $V(X)$ and $K$ be the kernel of the action $A$ on the set
of orbits of $N$. By considering the order of graph, we see that either $r=3$ or $r=p$. Assume first that
$N$ is an elementary abelian $3$-group. Then $X_N$ has order $p$ and has valency $2$ or $4$.
It is easy to see that $K$ acts faithfully on each orbits and therefore $K \leq S_3$, which implies that $|N|=3$.
Let $P$ be a Sylow $p$-subgroup of $A$. Then $|P|=p$ and since $N$ is a normal subgroup of $A$, it follows that
the orbits of $N$ form a complete block system of $A$. Moreover $P\times N$ acts regularly on $V(X)$ which implies that $X$
is a Cayley graph, a contradiction. Next assume that $N$ is an elementary abelian $p$-group. Then $X_N$
has order $3$ and its valency is $2$. Hence $\Aut(X_N)\cong D_6$. Suppose that $\Omega=\{\Delta_0, \Delta_1, \Delta_2\}$
and $\Delta_i \thicksim \Delta_{i+1}$ where the subscripts are taken modulo 3. If $\Delta_i$
has an edge, then $X[\Delta_i]\cong C_p$ and it is easy to see that $|K_v|\leq 2$ where $v\in V(X)$.
Thus we have $|K|\leq 2p$. In case $\Delta_i$ has no edge, then since $|V(X_N)|=3$ and $p\geq 11$ we conclude
that $X[\Delta_i \cup \Delta_{i+1}] \cong C_{2p}$. Let $K^*$ be the kernel of the action of $K$ on $X[\Delta_i\cup\Delta_{i+1}]$.
Now the connectivity of $X$ and transitivity of $A/K$ on $V(X_N)$ imply that $K^{*}=1$ and
consequently, $K\leq\Aut(X[\Delta_i\cup\Delta_{i+1}])\cong D_{4p}$.
Since $K$ fixes $\Delta_i$ it follows that $|K|\leq 2p$. Since $A/K$ is transitive on $V(X_N)$ it follows
that 3 divides $|A/K|$ and $A/K$ has an element of order $3$, say $K\alpha$, where $\alpha \in A$. Therefore
$\alpha^3 \in K$ and hence $o(\alpha^3)=1$, 2, $p$ or $2p$. If $\alpha^3=1$ then
$\langle\alpha\rangle.P$ acts transitively and so regularly on $V(X)$, a contradiction.
Also if $o(\alpha^3)=2$, $p$ or $2p$ then $\alpha^2$, $\alpha^p$ and $\alpha^{2p}$ have order
$3$. In each case there is an element of order 3, say $\beta$, in $A$ such that
$\langle\beta\rangle.P$ acts regularly on $V(X)$, another contradiction. Hence $A$ is not solvable.

Let $L=\Sol(A)$, $\Sigma$ be the set of orbits of $L$ on $V(X)$ and $K$ be the
kernel of the action of $A$ on the set of orbits of $L$.
Assume first that $L$ is not trivial. Since $K_v$ is a $\{2, 3\}$-group, it implies that $K_v$ is solvable
and by considering $K=LK_v$ we see that $K$ is solvable. By the definition of solvable radical we conclude that
$K=L$ and $A/L \leq\Aut(X_L)$. If $\deg(X_L)=0$ then $A/L$ is solvable which implies that $A$ is solvable, a contradiction.
Since $|X|=3p$, this implies that $\deg(X_L)\neq 1, 3$. In case $\deg(X_L)=2$ we have $\Aut(X_L)\cong D_{2n}$, where $|V(X_L)|=n$ and $n\in\{3, p \}$.
Now $A/L$ is solvable and so $A$ is solvable, which is impossible. Thus we may assume that $\deg(X_L)=4$
and $L\leq S_3$. Thus $X_L$ is a circulant graph and $X_L=\Cay(G, S)$, where
$G\cong\mathbb{Z}_p=\langle a \rangle$, $S=\{a^r, a^{-r}, a^s, a^{-s}\}$ where $(r, p)=(s, p)=1$.
By \cite[Theorem 1.2]{BFSX}, $X_L$ is normal. Hence $\Aut(G, S)=\Aut(X_L)$.
Since $\sigma: a \mapsto a^r$ is an automorphism of $G$ we have $S=\{a, a^{-1}, a^i, a^{-i}\}$ where $(i, p)=1$.

Set $B=\Aut(X_L)$ and $Y=X_L$. Let $B_1^{*}$ be the subgroup of $B_1$ which fixes $\{1\}\cup S$
and $Y_2(1)$ be the subgraph of $Y$ with vertex set
\begin{equation*}
\{1\}\cup S\cup \{a^2, a^{-2}, a^{i+1}, a^{-i+1}, a^{2i}, a^{-2i}, a^{-i-1}, a^{i-1}\}.
\end{equation*}
Also
\begin{equation*}
N_Y(a)=\{1, a^2, a^{i+1}, a^{-i+1}\}, N_Y(a^{-1})=\{1, a^{-2}, a^{-i-1}, a^{i-1}\}, N_Y(a^i)=\{1, a^{i+1}, a^{2i}, a^{i-1}\}
\end{equation*}
and $N_Y(a^{-i})=\{1, a^{-i+1}, a^{-2i}, a^{-i-1}\}$.
By a tedious computation we find that $B_1^{*}=1$ and $B_1$ have no any element of order $4$. Thus $B_1$ acts faithfully on $S$ and so $B_1 \leq S_4$.
Since $B_1$ has no any element of order 4, by considering the subgroups of $S_4$ we see that either
$B_1\cong \mathbb{Z}_2\times \mathbb{Z}_2$ or $|B_1|\in\{1, 2, 6\}$. Since $|B|=|B_1||G|$ we see that $|B|\in \{p, 2p, 4p, 6p\}$.
Moreover $\alpha: a \rightarrow a^{-1}$ is an element of $\Aut(G, S)$ and $o(\alpha)=2$. Hence
$|B_1|=2$ or $4$ or $6$. From the fact that $A/L$ acts transitively on $V(Y)$ we have $p\mid |A/L|$ and
$|A/L|\in \{p, 2p, 4p, 6p \}$. By considering the order of $B$ we see that $G\leq A/L$. So $G=T/L$
where $T\leq A$ and by Proposition \ref{p2}, $X$ is a Cayley graph, a contradiction.

Assume now that $L$ is trivial. Then $N$ is not solvable, where $N$ is the minimal normal subgroup of $A$.
If $N$ is semiregular then $|N|$ divides $3p$ and $N$ is solvable, a contradiction. Therefore $N$ is
not semiregular and the valency of quotient graph $X_N$ is
not $4$. By considering the order of graph we see that $X_N$ is not cubic.
Thus $\deg(X_N)=0$ or $2$ and since $A/K \leq \Aut(X_N)$ we conclude that $A/K$ is solvable.
Moreover, since $K=K_vN$ and $K_v$ is a $\{2, 3\}$-group we conclude that $K/N$ is solvable. From the fact that
$A/K \cong (A/N)/(K/N)$ we conclude that $A/N$ is solvable. We claim that $N$ is the unique minimal normal
subgroup of $A$. Suppose by contrary that $M\neq N$ is another minimal
normal subgroup of $A$. Since $MN/N \cong M/(M\cap N)\leq A/N$ we see that $M/(M \cap N)$ is
solvable. But $M \cap N=1$ and so $M$ is solvable, a contradiction. Hence $N$ is the only minimal
normal subgroup of $A$. Also since $|A|$ is not divided by $p^2$ we conclude that $N$ is a simple
group. Therefore $A$ is an almost simple group and $\soc(A)=N$. From the fact that $A$ is a $\{2, 3, p\}$-group,
we conclude that $N$ is a $K_3$-group. Also since $p\geq 11$ it follows that $N$ is isomorphic to either $\PSL_2(17)$ or $\PSL_3(3)$.

First assume that $N\cong\PSL_2(17)$. Then by part $(iii)$ of Lemma \ref{a1}, $A\cong\PSL_2(17)$ or $\PGL_2(17)$.
If $A\cong\PSL_2(17)$ then $|A:A_v|=51$. By \cite{CCNPW}, we see that $A$ has not any
subgroup of index $51$, a contradiction. Also if $A\cong\PGL_2(17)$, then by using \texttt{GAP} \cite{Gap},
$A$ has no any subgroup of index $51$, another contradiction.

Assume now that $N\cong\PSL_3(3)$. In this case, by part $(v)$ of Lemma \ref{a1}, $A\cong\PSL_3(3)$ or $A=\Aut(\PSL_3(3))$.
In the first case, by using \texttt{GAP} \cite{Gap}, and the fact that $A$ is a transitive permutation group of degree $39$,
we see that $A$ has a regular subgroup which means that $X$ is a Cayley graph, a contradiction. However the later case is impossible,
because there is no transitive permutation representation of $\Aut(\PSL_3(3))$ of degree $39$.
\end{proof}

Let $X$ be a connected tetravalent vertex-transitive graph of order $6p$. If $p=2$ or 3,
then by \cite{MR}, we see that $X$ is a Cayley graph. Also if $p=5$ or $p=7$ then by \cite{HRT} and using \cite{HR}
we see that $X$ is isomorphic to one of the graphs in the following lemma.

\begin{lemma}\label{l1}
A connected tetravalent vertex-transitive non-Cayley graph of order at most 42 are listed below:
\begin{itemize}
\item[$(i)$] The Coxeter graph;
\item[$(ii)$] the Desargues graph;
\item[$(iii)$] the dodecahedron graph;
\item[$(iv)$] one of the six graphs in Figure 1.
\end{itemize}
\end{lemma}
\begin{figure}[H]
\includegraphics[scale=0.5]{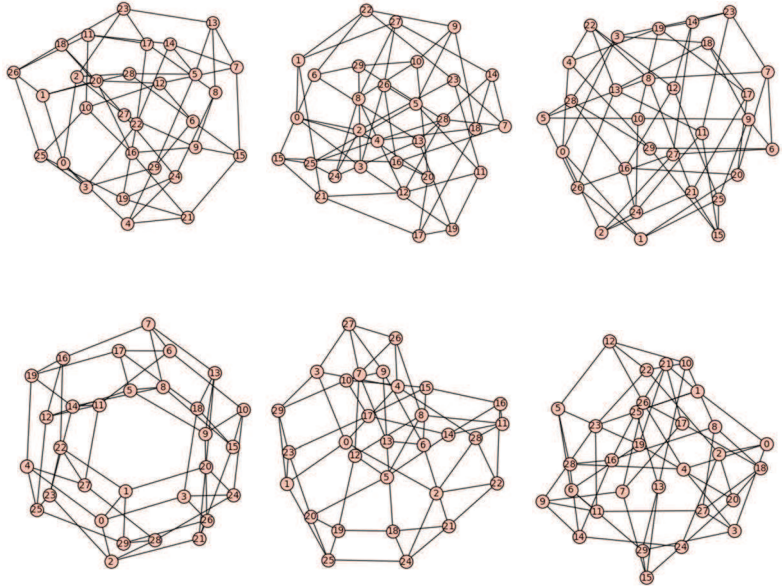}
\caption{Six graphs illustrated in Lemma \ref{l1}}\label{fig:case2}
\end{figure}

\textbf{Proof of the Theorem \ref{a3}.}
%\begin{theorem}\label{a3}
Let $X$ be a connected tetravalent vertex-transitive graph of order $6p$. If $p\leq 7$, then 
$X$ is one of the nine specified graphs in Lemma~\ref{l1}. So we may assume that $p\geq 11$.
%Then $X$ is a $\mathcal{VNC}$-graph if and only if $X\cong X_{3,p,t}$, where $1\leq t\leq p-1$ and $t^2\equiv -1$ (mod $p$).
%\end{theorem}
%\begin{proof}
Suppose that $X\cong X_{3,p,t}$, then by Proposition \ref{m1m2t}, $X$ is a $\mathcal{VNC}$-graph.
Assume that $X$ is a $\mathcal{VNC}$-graph. Let $A=\Aut(X)$ and $v\in V(X)$. Since $X$ is connected and tetravalent, $A_v$ is a $\{2, 3\}$-group
and so $|A|=2^s3^tp$ for some integers $s,t\geq 1$. Assume that $N$ is a minimal normal subgroup of $A$.
We know that $N=T^k$, where $T$ is a finite simple group. Let
$\Omega$ be the set of orbits of $N$ on $V(X)$ and $K$ be the kernel of the action of $A$ on $\Omega$.
Now we consider two cases:

{\bf Case I.} $A$ is solvable.

In this case $N$ ia an elementary abelian $r$-group, where $r\in\{2, 3, p\}$.
First suppose that $N$ is an elementary abelian $2$-group.
In this case $X_N$ has order $3p$ and valency $2$ or $4$. If $X_N$ has valency $2$
then $X_N \cong C_{3p}$ and $\Aut(X_N)\cong D_{6p}$. Suppose that
$\Omega=\{\Delta_0, \Delta_1, \Delta_2, \cdots,\Delta_{3p-1}\}$, where the subscripts
are taken modulo $3p$. Also suppose that $\Delta_i \thicksim \Delta_{i+1}$ in $X_N$. If
for some $i$, the induced graph $X[\Delta_i]$ has an edge then by considering the valency
of $X$, we see that $X[\Delta_{i-1}\cup\Delta_i]\cong C_{4}$ and $X[\Delta_i\cup\Delta_{i+1}]\cong K_{4}$
where $1\leq i\leq 3p-1$. Now by considering the induced graph $X[\Delta_0 \cup \Delta_{3p-1}]$,
we have a contradiction. Thus we may assume that each $X[\Delta_i]$ has no any edge. Thus
for each $i$, $X[\Delta_i\cup\Delta_{i+1}]\cong C_4$. It is easy to see
that in this case $X\cong C_{3p}[2K_1]$, which is a Cayley graph, a contradiction.

Assume now that $X_N$ has valency $4$. Then $K_v=1$ for each vertex $v$. By Proposition \ref{p3}, $X_N$ is a Cayley graph, unless
the case where $X_N$ is isomorphic to the line graph of the Petersen graph. Since $p\geq 11$,
$X_N$ is not isomorphic to the line graph of the Petersen graph. Hence $X_N$ is a Cayley graph
and by \cite[Proposition 1.2]{M.3}, is a $(3, p)$-metacirculant graph.
Thus $X_N$ has an automorphism $\sigma$ of order $p$ such that $\langle \sigma \rangle$
is semiregular on the vertex set of $X_N$. Moreover $X_N$ has an automorphism say $\tau$
such that $\tau$ normalizing $\langle\sigma\rangle$ and cyclically
permutes the $3$ orbits of $\langle \sigma \rangle$ and has a cycle of size $3$ in
its cycle decomposition. Thus, for some positive integer $i$, the group $\langle \sigma, \tau^i \rangle$, has order
$3p$ and acts regularly on $V(X_N)$. Therefore we may suppose that $X_N=\Cay(H, S)$ where $H=\langle \sigma, \tau^i \rangle$. First
suppose that $H\cong F_{3p}$. If $X_N$ is normal then
$\Aut(X_N)_1= \Aut(H, S)$. Since $X_N$ is connected and $\Aut(H, S)$ acts faithfully on $S$
we conclude that $\Aut(X_N)_1\leq S_4$. Moreover $S=\{x, x^{-1}, y, y^{-1}\}$ where $x$ and $y$ have not the same order. It is easy
to see that $\Aut(H, S)$ has no any element of order $3$ or $4$. Thus by considering
the subgroups of $S_4$ we see that $\Aut(X_N)_1=1$ or $\mathbb{Z}_2$ or $\mathbb{Z}_2 \times\mathbb{Z}_2$.
Therefore $|\Aut(X_N)| \leq 12p$ and so $|A/K|$ divides $12p$. By considering the order of
$\Aut(X_N)$ we see that $H\leq A/K$ and so $H=T/K$ for some $T\leq A$. Now $|T|=6p$ and
$T$ acts regularly on $V(X)$, a contradiction. Thus we may suppose that $X_N$ is not normal.
If $X_N$ is not primitive then by \cite[Theorem 3.1]{LX} we see that the valency of $X_N$ is not equal to $4$. Also if
$X_N$ is primitive then again by \cite[Lemma 2.1]{WX} we have a contradiction. Therefore $X_N$ is a Cayley
graph on $\mathbb{Z}_{3p}$.

Thus $X_N$ is a circulant graph and $X_N=\Cay(H, S)$, where $S=\{a^i, a^{-i}, a^j, a^{-j} \}$
and either $(i, 3p)=1$ or $(j, 3p)=1$ or $(i-j, 3p)=1$. Set $B=\Aut(X_N)$ and $Y=X_N$.
Let $B_1^{*}$ be the subgroup of $B_1$ which fixes $\{1\}\cup S$ and $Y_2(1)$ be the subgraph
of $Y$ with vertex set $\{1\} \cup S \cup (S^2-(1 \cup S))$. We see that
\begin{equation*}
S^2-(1\cup S)=\{a^{2i}, a^{i+j}, a^{i-j}, a^{-2i}, a^{-i+j}, a^{-i-j}, a^{2j}, a^{-2j}\}.
\end{equation*}
Also
\begin{equation*}
N_Y(a^i)=\{1, a^{2i}, a^{i+j}, a^{i-j}\}, N_Y(a^{-i})=\{1, a^{-2i}, a^{-i+j}, a^{-i-j}\},
N_Y(a^j)=\{1, a^{i+j}, a^{-i+j}, a^{2j}\}
\end{equation*}
and $N_Y(a^{-j})=\{1, a^{i-j}, a^{-i-j}, a^{-2j}\}$.
By a tedious computation we have $B_1^{*}=1$ and $B_1$ has no any element of order $4$. Thus $B_1$ acts faithfully
on $S$ and so $B_1 \leq S_4$. Since $B_1$ has no any element of order $3$, and the fact
that $\alpha: a\rightarrow a^{-1}$ is an element of $\Aut(H, S)$
and $o(\alpha)=2$, it implies that $B_1\cong\mathbb{Z}_2$ or $\mathbb{Z}_2\times\mathbb{Z}_2$ and so $|\Aut(X_N)|\in \{6p, 12p \}$.
Since $A/N$ acts transitively on $V(Y)$ we have that $3p$ divides $|A/N|$. Thus $|A/N| \in \{3p, 6p, 12p\}$.
By considering the order of $B$, we see that $H\leq A/N$ and so $H=T/N$ where $T\leq A$. Now
by Proposition \ref{p2}, $X$ is a Cayley graph, a contradiction.

Now suppose that $N$ is an elementary abelian $3$-group and
$\Omega=\{\Delta_0, \Delta_1, \Delta_2, \cdots, \Delta_{2p-1}\}$, where the
subscripts are taken modulo $2p$. Then $X_N$ has order $2p$ and so $X_N$ has
valency $2$, $3$ or $4$. First suppose that $X_N$ has valency $2$. Then $X_N\cong C_{2p}$
and ${\Aut(X_N)} \cong D_{4p}$. Suppose that $\Delta_i \thicksim \Delta_{i+1}$ in $X_N$.
If $\Delta_i$ has an edge then by considering the valency of $X$ we see that $N$ acts faithfully on $X[\Delta_i]$.
Hence $|N|=3$ and by considering the valency of $X$ we see that $X[\Delta_{i-1} \cup \Delta_i]\cong 3K_2$,
$|K_v|\leq 2$ and $|K|\leq 6$. Since $A/K$ is transitive on $V(X_N)$ it follows
that $2p$ divides $|A/K|$. Hence $A/K \cong \mathbb{Z}_{2p}$
or $A/K\cong D_{2p}$ or $A/K \cong D_{4p}$. If $A/K \cong \mathbb{Z}_{2p} $ or $A/K \cong D_{4p}$,
then $A/K$ has an element of order $2p$, say $K\alpha$, where $\alpha \in A$. Therefore
$\alpha^{2p}\in K$ and hence $o(\alpha^{2p})=1$, $2$, $3$ or $6$. If $\alpha^{2p}=1$ then
$\langle \alpha \rangle \times \langle \gamma \rangle$, where $\gamma$ is an element of
order $3$ in $K$, acts transitively and so regularly on $V(X)$, a contradiction. Also if
$o(\alpha^{2p})=2$, $3$ or $6$ then $\alpha^{2}$, $\alpha^3$ and $\alpha^{6}$ have order
$2p$. Thus in each case there is an element of order $2p$, say $\beta$, in $A$ such that
$\langle \beta \rangle \times \langle \gamma \rangle$ acts transitively and so regularly
on $V(X)$, a contradiction. Thus we may suppose that $A/K \cong D_{2p}$. Thus $A/K$ has
an element of order $p$, say $K\alpha$. Clearly, $K\alpha$ does not fix any element
of $V(X_N)$ and so $\alpha$ does not.
Clearly, $\alpha^p \in K$ and so $o(\alpha^p)=1$ or $2$ or $3$ or $6$. In each case we can
find an element of order $p$, say $\beta$ such that $\langle \beta \rangle$ acts
semiregularly on $V(X)$. Set $H=\langle \beta \rangle \times \langle \gamma \rangle$,
where $\gamma$ is an element of order $3$ in $K$, acts semiregularly  on $V(X)$ and
has two orbits on $V(X)$, which means that $X$ is a bi-Cayley graph over
$H\cong\Bbb Z_{3p}$. Then, by Proposition \ref{m1m2t}, $X\cong X_{3,p,t}$, where
$1\leq t\leq p-1$, and $t^2\equiv -1$ (mod $p$).

Suppose that $\Delta_i$ has no edge. Then it is easy to see that either
$X[\Delta_0\cup\Delta_{2p-1}]=X[\Delta_{i-1} \cup \Delta_i] \cong C_6$ $(1 \leq i \leq 2p-1)$
or $X[\Delta_{i} \cup \Delta_{i+1}] \cong 3K_2$ $(1 \leq i \leq 2p-1, i \hskip 0.1 cm  {\rm is} \hskip 0.1 cm {\rm odd})$
and $X[\Delta_{i} \cup \Delta_{i+1}] \cong K_{3, 3}$
$(0 \leq i \leq 2p-1, i \hskip 0.1 cm {\rm is} \hskip 0.1 cm {\rm even})$.

First suppose that $X[\Delta_0\cup \Delta_{2p-1}]=X[\Delta_{i-1} \cup \Delta_i] \cong C_6$ $(1 \leq i \leq 2p-1)$.
Then we consider the action of $K$ in $X[\Delta_i]$ $(0\leq i\leq 2p-1)$. Now it is easy to see that
$K$ acts faithfully on $X[\Delta_i]$ and so $K\leq S_3$. Since $2p$ divides $|A/K|$ and $|A/K|$ divides $4p$
with a similar arguments as above we get a contradiction. Now suppose that $X[\Delta_{i} \cup \Delta_{i+1}] \cong 3K_2$
$(1 \leq i \leq 2p-1, i \hskip 0.1 cm  {\rm is} \hskip 0.1 cm {\rm odd})$ and
$X[\Delta_{i} \cup \Delta_{i+1}] \cong K_{3, 3}$
$(0 \leq i \leq 2p-2, i \hskip 0.1 cm {\rm is} \hskip 0.1 cm {\rm even})$.

Let $V(\Delta_i)=\{a_i, b_i, c_i \}$ $(0 \leq i \leq 2p-1)$. Without loss of generality we may assume that
$a_i\thicksim a_{i+1}, b_i\thicksim b_{i+1}$ and $c_i\thicksim c_{i+1}.$ Now it is easy
to see that $\alpha: a_i \mapsto b_i, b_i \mapsto c_i, c_i \mapsto a_i$ and
$\beta:a_i \mapsto a_{i+2}, b_i\mapsto b_{i+2}, c_i\mapsto c_{i+2}$ are
automorphisms of $X$ of orders $3$ and $p$, respectively, where $(0\leq i \leq 2p-1)$ .
Clearly $\langle \alpha \rangle \times \langle \beta \rangle \cong \mathbb{Z}_{3p}$
acts semiregularly and has two orbits on $V(X)$. Therefore $X$ is a bi-Cayley graph
over the group $\mathbb{Z}_{3p}$ and again, by Proposition \ref{m1m2t}, $X=X_{3,p,t}$, where
$1\leq t\leq p-1$ and $t^2\equiv -1$ (mod $p$).

Suppose that $X_N$ has valency either $3$ or $4$. If $\Aut(X_N)$ is primitive
then by \cite[Corollary 6.6]{Ma}, $p\geq 313$, $2p=m^2+1$ for some composite integer
$m$. Also $\Gamma$ has valency either $\frac{m(m-1)}{2}$ or $\frac{m(m+1)}{2}$, a
contradiction. Thus we may suppose that $\Aut(X_N)$ is imprimitive. Now by
\cite[Theorem]{M.3}, $X_N$ is metacirculant. If $X_N$ has valency 3 then by
\cite[Main Theorem]{T1}, $X_N$ is isomorphic to generalized Petersen graph
$GP(p, k)$. By \cite[Theorems 1, 2]{FGW}, we see that $|\Aut(X_N)|=4p$.
If $\Delta_i$ has an edge then $X[\Delta_i]\cong C_3$ $(0\leq i \leq 2p-1)$.
We see that $X_N$ can not have valency $3$, a contradiction. Thus we may suppose
that $\Delta_i$ $(0\leq i \leq 2p-1)$ has no edges. Let $V(\Delta_i)=\{a_i, b_i, c_i \}$
$(0\leq i \leq 2p-1)$. Let $\Delta_{i+1}$, $\Delta_{i+2}$ and $\Delta_{i+3}$ are the
vertices adjacent to $\Delta_i$. Since $X_N$ has valency 3, we may suppose that
$X[\Delta_i \cup \Delta_{i+1}] \cong C_6$, $X[\Delta_{i} \cup \Delta_{i+2}]\cong 3K_2$
and $X[\Delta_{i} \cup \Delta_{i+3}] \cong 3K_2$. Now it is easy to see that
$\alpha: a_i \mapsto b_i, b_i \mapsto c_i, c_i \mapsto a_i$
is an automorphisms of $X$ of order $3$. Also $K$ acts faithfully on $X[\Delta_i]$
and so $K=N \cong \mathbb{Z}_{3}$. By \cite[Theorem 1.1]{AGh}, we see that the generalized
Petersen graph is a normal bi-Cayley graph on $H=\mathbb{Z}_{p}$. Clearly, $H$ acts
semiregularly on $V(X_N)$. If $H$ is not a subgroup of $A/K$ then $p^2 \mid |\Aut(X_N)|$,
a contradiction. Thus $H \leq A/K$ and $A/K$ has a semiregular element of order $p$, say
$K\alpha$. Now $\alpha^p \in K$ and so $o(\alpha^p)=1, 2$ or 3. In each case it is easy to
find an element of order $p$, say $\beta$, which is a power of $\alpha$. If $\beta$ fixes
an element then it is easy to see that $K\alpha$ or $K\alpha^2$ or $K\alpha^3$ fixes at
least one $\Delta_{i}$, a contradiction. Now $\langle \beta \rangle \times \langle \gamma \rangle \cong \mathbb{Z}_{3p}$
acts semiregularly and has two orbits on $V(X)$ where $\gamma\in K$ and $O(\gamma)=3$. Therefore $X$ is a bi-Cayley graph over
group $\mathbb{Z}_{3p}$ and by Proposition \ref{m1m2t}, $X=X_{3,p,t}$, where
$1\leq t\leq p-1$ and $t^2\equiv -1\pmod p$.
Also if $X_N$ has valency 4 then $N=K$. Since $A_v$ is $\{2, 3\}$-group we conclude that
$A$ has a semiregular element of order $p$, say $\delta$. Since $N_A(N)/C_A(N)$ is a subgroup
of $\Aut(N)\cong\mathbb{Z}_{2}$ we conclude that $\delta$ or ${\delta}^2 \in C_A(N)$. Also if
$N=\langle\gamma\rangle$ then $\langle\delta,\gamma\rangle$  or $\langle\delta^2,\gamma\rangle$ is a semiregular subgroup
of $\Aut(X)$ which is isomorphic to $\mathbb{Z}_{3p}$. Thus $X$ is a bi-Cayley graph over $H\cong\mathbb{Z}_{3p}=\langle a \rangle$
and by Proposition \ref{m1m2t}, $X=X_{3,p,t}$, where $1\leq t\leq p-1$ and $t^2\equiv -1 \pmod p$.

Finally, suppose that $N$ is an elementary abelian $p$-group. By considering the order
of $A$ we see that $|N|=p$. Suppose that  $N_v\neq 1$ for some $v \in V(X)$ and $w$ is an
another arbitrary vertex of $X$. Since $X$ is transitive, there exists $\alpha \in \Aut(X)$
such that $w^\alpha=v$ and so  $N_v=N_{w^\alpha}=\alpha^{-1}N_w\alpha=N_w$. Thus $N$ is
trivial, a contradiction. Therefore $N$ acts semiregularly on $V(X)$ and so $X$ is an
$N$-regular cover of $X_N$. By \cite[Page 173]{MR} we know that $X_N$ is a Cayley graph. Thus
$\Aut(X_N)$ has a subgroup which acts regularly on $V(X)$, say $G$. Suppose
that $\tilde{G}$ is a lift of $G$. Now it is easy to see that $\tilde{G}$ acts
regularly on $V(X)$ and so $X$ is a Cayley graph, a contradiction.

{\bf Case II.} $A$ is not solvable.

Let $L=\Sol(A)$ be the solvable radical of $A$ and $\Omega$ be the set of orbits
of $L$ on $V(X)$. First suppose that $L$ is not trivial. Let $K$ be the kernel of $A$
on $\Omega$. Since $K_v$ is a $\{2, 3\}$-group it implies that $K_v$ is solvable and so
by considering $K=LK_v$ we see that $K$ is solvable. Now by the definition of $L$ we
conclude that $K=L$ and $A/L \leq \Aut(X_L)$. If $\deg(X_L)=0$ then $A/L$ is solvable
and so $A$ is solvable, a contradiction. Also if $\deg(X_L)=2$ then $\Aut(X_L)\cong D_{2n}$,
where $|V(X_L)|=n$ and $n\in\{3, 6, p, 2p, 3p \}$. Now $A/L$ is solvable and so $A$ is solvable,
a contradiction. Thus we may suppose that $\deg(X_L)=3$ or $4$. If $\deg(X_L)=3$ then
$n\in \{6,  2p, 3p \}$. By \cite[Theorem 1.1]{LLW}, we see that $\Aut(X_L)$
is solvable and so $A/L$ is solvable. Now $A$ is solvable, a contradiction. Also if $\deg(X_L)=4$
then $L=K$ and $K$ acts faithfully on each orbits of $K$. Moreover, in this case
$n\in \{6, p, 2p, 3p \}$. If $n=6$, then $X_L$ is isomorphic to the octahedral graph and so $|\Aut(X_L)|=48$.
Thus $\Aut(X_L)$ is solvable which implies that $A$ is solvable, a contradiction. If $n=p$ then by $|V(X)|=|L|n$ we see that $|L|=6$ and $L \cong S_3$ or $\mathbb{Z}_6$. Also since $A_v$ is a
$\{2, 3\}$-group, it implies that $A$ has a semiregular element of order $p$, say $\alpha$.
If $L$ as subgroup of $S_6$ is isomorphic to $S_3$ then since $L$ is semiregular we conclude that
$S_3 \times \langle \alpha \rangle$ acts regularly on $V(X)$, a contradiction. Similarly if $L$
as subgroup of $S_6$ is isomorphic to $\mathbb{Z}_6=\langle \beta \rangle$ then
$\langle \alpha \rangle \times \langle \beta \rangle$ acts regularly on $V(X)$,
a contradiction. Also if $|V(X_L)|=2p$ then $|L|=3$.
Since $A_v$ is $\{2, 3\}$-group we conclude that $A$ has a semiregular element, say $\alpha$, of
order $p$, . Since $N_A(L)/C_A(L)$ is a subgroup of $\Aut(L)$ we
conclude that $\alpha \in C_A(L)$. Also if $\beta\in L$  and $\beta$ has order $3$
then $\langle \alpha, \beta \rangle$ is a semiregular subgroup of $\Aut(X)$ which
is isomorphic to $\mathbb{Z}_{3p}$. Thus $X$ is a bi-Cayley graph over $H$, where
$H\cong\mathbb{Z}_{3p}$. By Proposition \ref{m1m2t}, $X=X_{3,p,t}$, where
$1\leq t\leq p-1$ and $t^2\equiv -1 \pmod p$. Finally if $|V(X_L)|=3p$ then $|L|=2$.
By Proposition \ref{p3}, $X_N$ is a Cayley graph, unless
the case where $X_N$ is isomorphic to the line graph of the Petersen graph. Since $p\geq 11$, it follows that
$X_N$ is not isomorphic to the line graph of the Petersen graph. Also if $X_N$ is a Cayley
graph then $\Aut(X_N)$ has a subgroup say $H$ which acts regularly on $V(X_N)$. Thus $|H|=3p$
and by \cite[Proposition 1.2]{M.3}, $X_N$ is a $(3, p)$-metacirculant graph.
Thus $X_N$ has an automorphism $\sigma$ of order $p$ such that $\langle \sigma \rangle$
is semiregular on the vertex set of $X_N$, and an automorphism $\tau$ normalizing $\langle\sigma\rangle$ and cyclically
permuting the $3$ orbits of $\langle \sigma \rangle$ such that $\tau$ has a cycle of size $3$ in
its cycle decomposition. Thus $\langle \sigma, \tau^i \rangle$ for some positive integer $i$, has order
$3p$ and acts regularly on $V(X_N)$. Therefore we may suppose that $X_N=\Cay(H, S)$. First
suppose that $H=\langle \sigma, \tau^i \rangle \cong F_{3p}$. If $X_N$ is normal then
$\Aut(X_N)_1= \Aut(H, S)$. Since $\Aut(H, S)$ acts faithfully on $S$
we conclude that $\Aut(X_N)_1\leq S_4$. Also since $X_N$ is connected we conclude
that $S=\{x, x^{-1}, y, y^{-1}\}$ and $x$ and $y$ have not the same orders. Now it is easy
to see that $\Aut(H, S)$ has no any element of order $3$ or $4$. Thus by considering
the subgroups of $S_4$ we see that $\Aut(X_N)_1=1$ or $\mathbb{Z}_2$ or $\mathbb{Z}_2 \times\mathbb{Z}_2$.
Therefore $|\Aut(X_N)| \leq 12p$ and so $|A/K|$ divides $12p$. Now by considering the order of
$\Aut(X_N)$ we see that $H\leq A/K$ and so $H=T/K$ for some $T\leq A$. Now $|T|=6p$ and
$T$ acts regularly on $V(X)$, a contradiction. Thus we may suppose that $X_N$ is not normal.
If $X_N$ is not primitive then by \cite[Theorem 3.1]{LX} we see that the valency of $X_N$ is not equal to $4$. Also if
$X_N$ is primitive then again by \cite[Lemma 2.1]{WX} we have a contradiction. Therefore $X_N$ is a Cayley
graph on $\mathbb{Z}_{3p}$. Now similar to the case where $A$ is solvable, we conclude that $X$
is a Cayley graph, a contradiction.

Suppose that $L$ is trivial. By considering the order of $A$, we see that
$N$ is a $K_3$-group. Since $p\geq 11$, it follows that $N\cong \PSL_2(17)$ or $\PSL_3(3)$.
If $N$ has more than two orbits then $n\in\{3, 6, 2p, 3p\}$. First assume that $n=2p$
or $3p$. We know that $A/K$ acts transitively on $V(X_N)$ and $K=K_{\alpha}N$, where $v\in V(X)$.
Thus $p^2$ divides $|A|$, a contradiction. Thus we may suppose that $n=3$ or $6$.
By using \texttt{GAP} \cite{Gap}, it is easy to see that $\PSL_2(17)$ and $\PSL_3(3)$ does not have
any orbit of order $p$ or $2p$, where $p\in \{13, 17\}$. Therefore $N$ has at most two orbits.
Now $\Aut(X_N)$ is solvable and so $A/K$ is solvable. Also since $K_v$ is $\{2, 3\}$-group
and $K=K_vN$, it follows that $K/N$ is solvable. Since $A/K \cong (A/N)/(K/N)$ we conclude that
$A/N$ is also solvable. Suppose that $M$ is another minimal normal subgroup of $A$. Thus by the fact that
$MN/N \cong M/(M \cap N)\leq A/N$ we see that $M/(M\cap N)$ is solvable. But $M\cap N=1$
and so $M$ is solvable, a contradiction. Hence $N$ is the only minimal normal subgroup of $A$.
Since $|A|$ is not divisible by $p^2$ we conclude that $N$ is a simple group. Therefore $A$
is an almost simple group and $\soc(A)=N$. Since $A$ is $\{2, 3, p\}$-group, it follows that
$N$ is a $K_3$-group. Also since $p\geq 11$ we conclude that $N\cong \PSL_2(17)$ or $N \cong \PSL_3(3)$.

First assume that $N \cong \PSL_2(17)$. Then $A\cong \PSL_2(17)$ or $\PGL_2(17)$.
If $A\cong \PSL_2(17)$ then $|A:A_v|=102$. By \cite{CCNPW}, we see that $A$ has two
subgroups of index $102$ and both of them have an element of order $4$. Also we see that $A_v$
acts transitively on the set of vertices adjacent to $v$. By Proposition \ref{p2.2}, $X$
is a symmetric graph. If $X$ is one-regular then $|A|=2^3.3.17$, a contradiction.
By \cite[Theorem 1]{BGh}, we see that $X\cong C_{3p}[2K_1]$ and so $X$ is a Cayley graph, a contradiction.
Also if $A\cong \PGL_2(17)$ then by using \texttt{GAP} \cite{Gap}, $A$ has no any subgroup of index $102$, a contradiction.

Assume that $N \cong \PSL_3(3)$. Then $A\cong \PSL_3(3)$ or $A=\Aut(\PSL_3(3))$.
Let $A=\Aut(\PSL_3(3))$. Since $A\leq S_{78}$ is a transitive permutation group of degree
$78$, there exists a core-free subgroup $H$ of $A$ of index $78$ such that $A$ is isomorphic to the permutation representation of the
action of $A$ on the right cosets of $H$ in $G$ by right multiplication. One can check, by using \texttt{GAP} \cite{Gap}, that in this case $A$
has a regular subgroup of order $78$ which means that $X$ is a Cayley graph, a contradiction. Hence
$A=\PSL_3(3)$ and $X=\cos(\PSL_3(3),H,D)$, is a coset graph, where $H\cong (\mathbb{Z}_3\times\mathbb{Z}_3)\rtimes Q_8$ and
$D$ is a union of double cosets of $H$ in $G$. Since $X$ is $4$-valent, we have $|D|/|H|=4$.
By using \texttt{GAP} \cite{Gap}, there are $15$ possibilities for $D$ and in each case $\langle D\rangle$
is a proper subgroup of $\PSL_3(3)$ which means that $X$ is disconnected, a contradiction. This complete the proof.
%\end{proof}

\bigskip
\bigskip
{\footnotesize \pn{\bf M.~Arezoomand}\;
\\ {University of Larestan, 74317-16137,} {Larestan, I. R. Iran}\\
{\tt Email: arezoomand@lar.ac.ir}\\

{\footnotesize \pn{\bf M.~Ghasemi}\; \\ {Department of
Mathematics},\\ {Urmia University, 57135,} {Urmia, I. R. Iran}\\
{\tt Email: m.ghasemi@urmia.ac.ir}\\

{\footnotesize \pn{\bf Mohammad~A.~Iranmanesh}\; \\ {Department of
Mathematics},\\ {Yazd University,  89195-741,} { Yazd, I. R. Iran}\\
{\tt Email: iranmanesh@yazd.ac.ir}

\begin{thebibliography}{100}
%\bibitem{AS}
%B. Alspach, R.J. Sutcliffe, Vertex-transitive graphs of order $2p$, \emph{ Ann New York Acad Sci.} \textbf{319} (1979) 18--27.

%\bibitem{AX}
%B. Alspach, M.Y. Xu, $1/2$-transitive graphs of order $3p$, \emph{J. Algebr. Combin.} \textbf{3} (1994) 347--355.

\bibitem{AGh}
M. Arezoomand, M. Ghasemi, Normality of one-matching semi-Cayley graphs
over finite abelian groups with maximum degree three, \emph{Contrib. Discrete. Math.} \textbf{15(3)} (2020) 75--87. 

%\bibitem{AT}
%M. Arezoomand, B. Taeri, Isomorphisms of finite semi-Cayley graphs, Acta Math. Sin. Engl. Ser. \textbf{31}(4) (2015) 715-730.

%\bibitem{AT2}
%M. Arezoomand, B. Taeri, Normality of 2-Cayley digraphs, Discrete Math. \textbf{338} (2015) 41-47.

\bibitem{BGh}
K. Ber\v ci\v c, M. Ghasemi, Tetravalent arc-transitive graphs of order twice a product of two primes, \emph{Discrete Math.} \textbf{312} (2012) 3643--3648.

\bibitem{BFSX}
Y.G. Baik, Y.-Q. Feng, H.S. Sim, M.Y. Xu, On the normality of Cayley
graphs of abelian groups,  \emph{Algebra Colloq.}  \textbf{5} (1998)  297--304.
\bibitem{B}
N.L. Biggs, Algebraic Graph Theory, 2nd edn., Cambridge Mathematica Library, Cambridge University Press, Camberidge, 1986.

%\bibitem{C}
%C.Y. Chao, On the classification of symmetric graphs with a prime number of vertices, \emph{Trans. Amer. Math. Soc.}
%\textbf{158} (1971) 247–-256.

\bibitem{CGhQ}
H.W. Cheng, M. Ghasemi, S. Qiao, Tetravalent vertex-transitive graphs of order twice a prime square, \emph{Graph. Combin.} \textbf{32} (2016) 1763--1771.

\bibitem{CO}
Y. Cheng, J. Oxley, On weakly symmetric graphs of order twice a
prime, \emph{ J. Combin. Theory B}  \textbf{42} (1987) 196--211.

\bibitem{CCNPW}
J.H. Conway, R.T. Curties, S.P. Norton, R.A. Parker, R.A. Wilson,  \emph{Atlas of Finite Groups}, Clarendon Press, Oxford, 1985.

%\bibitem{DKM}
%E. Dobson, I. Kovacs and S. Miklavi\v{c},
%The automorphism groups of non-edge-transitive
%rose window graphs, \emph{Ars Math. Contemp.} \textbf{9} (2015) 63-–75.

\bibitem{Feng}
Y.-Q. Feng, On vertex-transitive graphs of odd prime-power order,
\emph{Discrete Math.} \textbf{248} (2002) 265--269.

\bibitem{FGW}
R. Frucht, J. E. Graver, M.E. Watkins, The groups of the generalized Petersen graphs, \emph{Math. Proc. Camb.
Phil. Soc.}  \textbf{70} (1971) 211--218.

\bibitem{Gap}
$\texttt{GAP}$ Team, Group, $\texttt{GAP}$-Groups, Algorithms, and Programming, Version 4.5.5, 2012, http://www.gapsystem.org.

\bibitem{G}
D. Gorenstein, \emph{ Finite simple groups } Plenum Press, New York, 1982.

\bibitem{HIP}
A. Hassani, M. A. Iranmanesh, C. E. Praeger, On vertex-imprimitive
graphs of order a product of three distinct odd primes, \emph{J.
Combin. Math. Combin. Comput.} \textbf{28} (1998) 187--213.

\bibitem{HRT}
 D. Holt, G. Royle, G. Tracey, The transitive groups of degree 48 and some applications,
\emph{accepted to J. Algebra} 

\bibitem{IP}
M.A. Iranmanesh, C.E. Praeger, On non-Cayley vertex-transitive graphs of order a product of three primes, \emph{ J. Combin. Theory B}  \textbf{81} (2001) 1--19.

%\bibitem{KKMW} I. Kov\'{a}cs, B. Kuzman, A. Malni\v{c}, S. Wilson, Characterization of edge-transitive 4-valent
%bicirculants, J. Graph Theory.

\bibitem{L}
C.H. Li, A. Seress, On vertex-transitive non-Cayley graphs of
square-free order, \emph{Designs, Codes and Cryptography.} \textbf{34} (2005)
265--281.

\bibitem{LLW}
C.H. Li, Z.P. Lu, G.X. Wang, Vertex-transitive cubic graphs of square-free order, \emph{J. Graph
Theory.} \textbf{75} (2014) 1--19.

\bibitem{LX}
Z.P. Lu, M.Y. Xu, On the normality of Cayley graphs of order $pq$, \emph{Austral. J. Combin.} \textbf{27} (2003) 81--93.

\bibitem{Ma}
D. Maru\v si\v c,  On vertex symmetric digraphs, \emph{Discrete Math.}
\textbf{36} (1981) 69--81.

\bibitem{M.1}
D. Maru\v si\v c, Cayley properties of vertex symmetric graphs,
\emph{Ars Combin.} \textbf{16B} (1983) 297--302.

\bibitem{M.2}
D. Maru\v si\v c, Vertex transitive graphs and digraphs of order
$p^k$, \emph{Ann. Discrete Math.} \textbf{27} (1985) 115--128.

%\bibitem{MP}
%D. Maru\v si\v c, T. Pisanski, Symmetries of hexagonal graphs on the torus. \emph{Croat. Chemica Acta}, \textbf{73} (2000)  969–-981.

\bibitem{M.3}
D. Maru\v si\v c, R. Scapellato, Characterizing vertex-transitive
$pq$-graphs with an imprimitive automorphism subgroup, \emph{J. Graph
Theory.} \textbf{16} (1992) 375--387.

\bibitem{M.4}
D. Maru\v si\v c, R. Scapellato, Classifying vertex-transitive
graphs whose order is a product of two primes, \emph{Combinatorica.} \textbf{14}
(1994) 187--201.

\bibitem{M.5}
D. Maru\v si\v c, R. Scapellato, B. Zgrabli\v c, On
quasiprimitive $pqr$-graphs, \emph{Algebra Colloq.} \textbf{2} (1995) 295--314.

%\bibitem{M}
%B.D. McKay, Transitive graphs with fewer than $20$ vertices, \emph{Math. Comp.} \textbf{33} (1979) 1101--1121.

\bibitem{MC}
B.D. McKay, C.E. Praeger, Vertex-transitive graphs which are not
Cayley graphs I, \emph{J. Austral. Math. Soc.} \textbf{56} (1994) 53--63.

\bibitem{MC1}
B.D. McKay, C.E. Praeger, Vertex-transitive graphs which are not
Cayley graphs II, \emph{J. Graph Theory} \textbf{22} (1996) 321--334.

\bibitem{MR}
B.D. McKay, G. F. Royle, The Transitive Graphs with at Most 26 Vertices. \emph{Ars Combin.} \textbf{30} (1990) 161--176.
%B.D. McKay, G. Royle, Cubic transitive graphs, http://units.maths.uwa.edu.au/gordon
%/remote/cubtrans/index.html.

\bibitem{Mill}
A.A. Miller, C.E. Praeger, Non-Cayley vertex-transitive graphs of
order twice the product of two odd primes, \emph{J. Algebraic Combin.} \textbf{3}
(1994) 77--111.

%\bibitem{P}
%C.E. Praeger, Finite normal edge-transitive Cayley graphs, \emph{Bull. Austr. Math. Soc.} \textbf{60} (1999) 207--220.

%\bibitem{PX}
%C.E. Praeger, M.Y. Xu, Vertex-primitive graphs of order a product of two distinct primes, \emph{J. Combin. Theory B} \textbf{59} (1993) 245--266.

\bibitem{HR}
 Private communication with Holt and Royle.

\bibitem{QZ} S. Qiao, J. X. Zhou, On tetravalent vertex-transitive bicirculants, Indian J. Pure Appl. Math. 51(1) (2020) 277-288.

\bibitem{Seress}
A. Seress, On vertex-transitive non-Cayley graphs of order $pqr$,
\emph{Discrete Math.} \textbf{182} (1998) 279--292.

\bibitem{S1}
B.O. Sabidussi, Vertex-transitive graphs, Monash Math. \textbf{68} (1964) 426--438.

\bibitem{S}
G. Sabidussi, On a class of fix-point-free graphs, \emph{Proc. Amer.
Math. Soc.} \textbf{9} (1958) 800--804.

\bibitem{T1}
N.D. Tan, Cubic $(m, n)$-metacirculant graphs Cayley graphs which are not Cayley graphs,   \emph{Discrete Math.} \textbf{154} (1996) 237--244.

%\bibitem{T2}
%N.D. Tan, On the classification problem for tetravalentmetacirculant graphs, \emph{J. Discrete Math. Sci.  Crypt.} \textbf{8:3} (2005)   403--412.

\bibitem{WX}
R.J. Wang, M.Y. Xu, A classification of symmetric graphs of order $3p$, \emph{J. Combin. Theory B}  \textbf{58} (1993) 197--216.

\bibitem{W}
H. Wielandt, Finite Permutation Groups, Acadamic Press, New York 1964.

\bibitem{Xu1}
M.Y. Xu,  Automorphism groups and isomorphisms of Cayley digraphs, \emph{Discrete Math.}  \textbf{82} (1998) 309--319.

%\bibitem{XX}
%J. Xu, M.Y. Xu, Arc-transitive Cayley graphs of valency at most four
%on abelian groups. \emph{Southeast Asian Bull. Math.} \textbf{25} (2001) 355--363.

%\bibitem{ZZ}
%M.M. Zhang, J.X. Zhou, Trivalent dihedral and bi-dihedrant. \emph{Ars Math. Contemp.} \textbf{21} (2021) $\sharp$P2.02

\bibitem{Z}
J.-X. Zhou, Tetravalent vertex-transitive graphs of order $4p$, \emph{J.
Graph Theory} \textbf{71} (2012) 402--415.

\bibitem{Z1}
J.-X. Zhou, Cubic vertex-transitive graphs of order $4p$
(Chinese), \emph{J. Sys. Sci. Math. Sci.} \textbf{28} (2008) 1245--1249.

\bibitem{Z2}
J.-X. Zhou, Cubic vertex-transitive graphs of order $2p^2$
(Chinese), \emph{Advance in Math.} \textbf{37} (2008) 605--609.

\bibitem{Z3}
J.-X. Zhou, Y.-Q. Feng, Cubic vertex-transitive graphs of order
$2pq$, \emph{J. Graph Theory} \textbf{65} (2010) 285--302.

%\bibitem{ZF} J.-X. Zhou, Y.-Q. Feng,, Cubic bi-Cayley graphs over abelian groups, \emph{European J. Combin.} {\bf 36} (2014) 679--693.
\end{thebibliography}
\end{document}